\documentclass[reqno]{amsart}
\usepackage{amsmath, amssymb, amsthm, epsfig}
\usepackage{hyperref, latexsym}
\usepackage{url}
\usepackage[mathscr]{euscript}
\usepackage{color}
\usepackage{setspace}
\usepackage{harpoon}
\usepackage{url}
\usepackage{refcheck} 
\norefnames
\nocitenames

\def\today{\ifcase\month\or
  January\or February\or March\or April\or May\or June\or
  July\or August\or September\or October\or November\or December\fi
  \space\number\day, \number\year}

 \newtheorem{theorem}{Theorem}
  
 \newtheorem{lemma}[theorem]{Lemma}
 
 \newtheorem{corollary}[theorem]{Corollary}
 \theoremstyle{definition}

 \theoremstyle{remark}

 \newcommand{\mc}{\mathcal}

 \newcommand{\F}{\mc{F}}

 \newcommand{\I}{\mc{I}}
 
 \newcommand{\J}{\mc{J}}

 \newcommand{\C}{\mathbb{C}}
 \newcommand{\CL}{\mc{C}}
 \newcommand{\R}{\mathbb{R}}

 \newcommand{\dx}{\text{\rm d}x}

    \renewcommand{\d}{\text{\rm d}}

\newcommand{\Rep}{\textrm{Re}}

\newcommand{\es}[1]{\begin{equation}\begin{split}#1\end{split}\end{equation}}
\newcommand{\est}[1]{\begin{equation*}\begin{split}#1\end{split}\end{equation*}}

\newcommand{\re}{{\rm Re}\,}

\newcommand{\<}{<\hspace{-1.5 mm}<}

\newcommand{\bo}{\boldsymbol}

\newcommand{\om}{\omega}
\newcommand{\la}{\lambda}
\newcommand{\ga}{\gamma}

\begin{document}




\title[]{A Central Limit Theorem for Operators}
\author[Gon\c{c}alves]{Felipe Gon\c{c}alves}
\date{\today}
\address{IMPA - Instituto de Matem\'{a}tica Pura e Aplicada - Estrada Dona Castorina, 110, Rio de Janeiro, RJ, Brazil 22460-320}
\email{ffgoncalves@impa.br}
\urladdr{www.impa.br/$\sim$lipe239}
\allowdisplaybreaks
\numberwithin{equation}{section}

\begin{abstract}
We prove an analogue of the Central Limit Theorem for operators. For every operator $K$ defined on $\mathbb{C}[x]$ we construct a sequence of operators $K_N$ defined on $\mathbb{C}[x_1,...,x_N]$ and demonstrate that, under certain orthogonality conditions, this sequence converges in a weak sense to an unique operator $\mathcal{C}$. We show that this operator $\mathcal{C}$ is a member of a family of operators $\mathfrak{C}$ that we call {\it Centered Gaussian Operators} and which coincides with the family of operators given by a centered Gaussian Kernel. Inspired in the approximation method used by Beckner in \cite{Be} to prove the sharp form of the Hausdorff-Young inequality, the present article shows that Beckner's method is a special case of a general approximation method for operators. In particular, we characterize the Hermite semi-group as the family of Centered Gaussian Operators associated with any semi-group of operators.
\end{abstract}

\maketitle

\section{Introduction}

\subsection{Background}

In the remarkable paper \cite{Be}, Beckner proved a sharp form of the Hausdorff--Young inequality for the Fourier Transform by reducing the problem to a hyper-contractive estimate associated with the Hermite semi-group. He proved that the operator norm of the Fourier transform $\F:L^p(\R)\to L^{p'}(\R)$ ($1<p\leq 2$) is attained at Gaussian functions if and only if the following semi-group operator
\es{\label{herm-semi-group}
T_\om: H_\ell(x)\mapsto \om^\ell H_\ell(x)
}
with $\om=i\sqrt{p-1}$ defines a contraction from $L^p(\R,\d\ga)$ to $L^{p'}(\R,\d\ga)$, where $\d\ga$ is the normal distribution on the real line and $\{H_\ell(x)\}$ is the set of Hermite polynomials associated with $\d\ga$.
\smallskip

To prove this contraction estimate he proposed a new type of approximation method. Using the Central Limit Theorem for the following two-point probability measure 
$$
\d\nu=\frac{\delta_{-1}+\delta_{1}}{2}
$$
in conjunction with a special iterative method, he was able to show that the desired contraction estimate for the Hermite semi-group is a consequence of an analogous contraction estimate for the two-point space and the operator
\est{
K_\om: \left\{
\begin{array}{lc}
\displaystyle 1\mapsto 1 \\
x\mapsto \om x.
\end{array}
\right.
}
That is
$$
\|K_\om f\|_{L^{p'}(\R,\d\nu)}\leq \|f\|_{L^p(\R,\d\nu)}
$$
for all $f\in\C[x]$. The above inequality is known as Beckner's two-point inequality. More details about his proof are left to Section \ref{concluding-rem}.
\smallbreak

Inspired by Beckner's approach, the present article demonstrates that Beckner's method is a special instance of a general approximation method (Theorem \ref{main-thm}) that we see as an analogue of the Central Limit Theorem for operators and which leads to (Theorem \ref{main-cor}) a {\it transference principle} for operators and hyper-contractive estimates. 
\smallbreak

Our main result, Theorem \ref{main-thm}, shows that for any given standardized probability measure $\d\alpha$ defined on $\R$ (i.e. $\d\alpha$ has zero mean, unit variance and finite moments of all orders) and any linear operator $K$ defined in $\C[x]$ satisfying a orthogonality condition, the sequence of operators $K_N$ defined in Section \ref{iig-seq} converges in a weak sense to an unique operator $\CL$, also defined in $\C[x]$, that belongs to a particular family of operators denoted by $\mathfrak{C}$ that we call {\it Centered Gaussian Operators}. As a particular case, we show that if $K_\om$ is a semi-group operator associated with the orthogonal polynomials generated by a given probability measure $\d\alpha$, then the mentioned orthogonality conditions are met and the Centered Gaussian Operator associated with $K_\om$ is Hermite semi-group operator $T_\om$ defined in \eqref{herm-semi-group}.
\smallskip

In the next sub-sections we define precisely the concepts involved and we state the main results. In Section \ref{pre} we prove some useful estimates and establish a representation theorem for every operator $\CL \in\mathfrak{C}$ in terms of Hermite polynomials. In Section \ref{proof-main-res} we give proofs to the main results. Finally, in Section \ref{concluding-rem} we explain how our results generalize Beckner's approximation method and also that the family of Centered Gaussian Operators $\mathfrak{C}$ can be identified with the family of operators on the real line given by a centered Gaussian kernel as studied in \cite{Ep,Li}.

\subsection{Notation}\label{notation}
Here we define the notation used throughout the paper. We use the word {\it standardized} to say that a given probability measure has zero mean, unit variance and finite moments of all orders. We denote by 
$$
\d\ga(x)=(2\pi)^{-1/2}\exp(-x^2/2)\dx
$$
the normal distribution. For a given probability measure $\d\alpha$ defined on $\R$ and every positive integer $N$ we denote by 
$$
\d\alpha_{N}(\bo x) = \d\alpha(x_1)\times\ldots\times\d\alpha(x_N)
$$ 
(with a sub-index) the $N$-fold product of $\d\alpha$ with itself which is defined in $\R^N$. On the other hand, we denote by 
$$
\d\alpha^{N}(x) = \underbrace{ \d\alpha*\ldots*\d\alpha(x)}_{N \,\text{times}}
$$
(with a super-index) the $N$-fold convolution of $\d\alpha$ with itself defined on $\R$. We always use bold letters to denote $N$--dimensional vectors when convenient, for instance $\bo x=(x_1,\ldots,x_N)$ or $\bo y=(y_1,\ldots,y_N)$. 

Given a function $f(x)$ defined for real $x$ we write 
$$
f_+(\bo x)=f(x_1+\ldots+x_N)
$$
(the dimension $N$ will be clear by the context). We also denote by $\C[x]$ the ring of polynomials with complex coefficients and by $\C[x_1,\ldots,x_N]$ the several variables analogue.

\subsection{Main Results}

\subsubsection{A sequence of operators generated by independently applying a given operator} \label{iig-seq}
Let $\d\alpha$ be a standardized probability measure, $q> 1$ and $K:\C[x]\to L^q(\R,\d\alpha)$ be a linear operator. For a given integer $N>0$ we define a linear operator $K_{N}:\C[x_1,\ldots,x_N]\to L^q(\R^N,\d\alpha_N(\sqrt{N}\bo x))$ as follows
\es{\label{approx-seq}
K_N= S_{N,\sqrt{N}}K_{N,N}K_{N,N-1}\ldots K_{N,1}S_{N,1/\sqrt{N}},
}
where $K_{N,n}$ denotes the restriction of the operator $K$ to the $n$th variable and
\es{\label{scaling-op}
S_{N,\lambda}:f(x_1,\ldots,x_N)\mapsto f(\lambda x_1,\ldots,\lambda x_N)
}
is a scaling operator defined for all $\la\in\C$. In particular, if $p_j(x)=x^j$ for all real $x$ and $f(\bo x)=p_{j_1}(x_1)\ldots p_{j_N}(x_N)$ we have
$$
K_N(f)(\bo x) = \frac{K(p_{j_1})(\sqrt{N}x_1)\ldots K(p_{j_N})(\sqrt{N}x_N)}{N^{\tfrac{j_1+\ldots+j_N}{2}}}.
$$

The sequence $\{K_N\}_{N> 0}$ defined by \eqref{approx-seq} is of a special type, it is generated by independently applying the given operator $K$ in each variable, resembling the process of convolving a measure with itself or, in the point of view of probability theory, of making a normalized sum of random variables.

\subsubsection{The family of Centered Gaussian Operators}\label{family-CL}
Here we define a family of operators that we call Centered Gaussian Operators. Let $\{H_\ell(x)\}_{\ell\geq0}$ denote the sequence of Hermite polynomials associated with $\d\ga(x)$ (see Section \ref{herm-poly}). The Hermite semi-group is a family of operators parametrized by $\omega\in\C$ and defined by 
\est{
T_w: H_\ell(x)\mapsto \omega^\ell H_\ell(x).
}
Often this semi-group is denoted by $e^{-zH}$ where $\om=e^{-z}$. 

We also need two other operators: the one-dimensional scaling operator $S_\lambda=S_{1,\la}$ already defined in \eqref{scaling-op} and a multiplication operator defined below
\est{
M_\tau:f(x)\mapsto {\sqrt{1+\tau}}e^{-\tau x^2/2}f(x),
}
for $\re \tau > -1$. This is a technical condition which guaranties that $M_\tau(\C[x])\subset L^q(\R,\d\ga)$ for some $q>1$.

The family of Centered Gaussian operators will be denoted by $\mathfrak{C}$ and defined by
$$
\mathfrak{C} = \left\{ M_\tau T_\omega S_\lambda: \la,\om,\tau\in\C, \,\,\,\, \Rep\,\tau>
-1\right\}.
$$
In Section \ref{concluding-rem} we shall explain how this family coincides with the family operators given by centered Gaussian kernels. The following is the main result of this article.

\begin{theorem}\label{main-thm}
Let $\d\alpha$ be a standardized probability measure defined on $\R$, let $q>1$ be a real number and $K:\C[x]\to L^q(\R,\d\alpha)$ be a linear operator. Define the numbers
$$
K_{\ell,m}=\int_\R K(H_\ell)(x)H_m(x)\d\alpha(x).
$$
Assume that 
\begin{enumerate}
\item $K_{0,0}=1$ and $K_{0,1}=K_{1,0}=0$,
\item $\re K_{0,2}>-1$.
\end{enumerate}
Then there exists a unique operator $\CL\in \mathfrak{C}$ such that
\es{\label{central-limit}
\lim_{N\to\infty} \int_{\R^N}K_N(f_+)(\bo x)g_+(\bo x)\d\alpha_N(\sqrt{N}\bo x) = \int_\R \CL(f)(x)g(x)\d\ga(x)
}
for every $f,g\in\C[x]$, where $K_N$ is the sequence of operators defined by \eqref{approx-seq}. Furthermore, the representation $\CL=M_\tau T_\omega S_\lambda$ is valid if and only if

\begin{enumerate} 
\item[$(i)$] $\tau=\frac{-K_{0,2}}{1+K_{0,2}}$
\item[$(ii)$]$ \la^2=1+K_{2,0}+ \tau(1+\tau) K_{1,1}^2 $
\item[$(iii)$] $\la\om=(1+\tau)K_{1,1}$.
\end{enumerate}
\end{theorem}
\noindent {\it Remarks.} 
\begin{enumerate}
\item Observe that if $\la^2\neq 0$ then the system $(i)$--$(iii)$ always has two solutions of the form $(\tau,\pm \omega,\pm \lambda)$. However, these two triples define the same operator $\CL$ since $T_\om S_\la = T_{-\om}S_{-\la}$. If $\la^2=0$ then the operator $\CL$ is still uniquely defined since $T_\om S_0=S_0$ for every $\om\in\C$. We also note that \ref{rep-lemma} $T_\om S_\la=S_b T_a$ if $ab=\om\la$ and $\la^2(1-\om^2)=1-a^2$.
\smallskip

\item We notice that assumptions $K_{0,0}=1$ and $K_{0,1}=K_{1,0}=0$ are necessary for the existence and non-vanishing of the limit at \eqref{central-limit} for $f(x)$ and $g(x)$ of the form $ax+b$. We also note that by equation $(i)$, condition $(2)$ is equivalent to $\Rep\, \tau>-1$. Condition $(1)$ is what we call {\it orthogonality condition}.
\end{enumerate}

\begin{theorem}[Transference Principle]\label{main-cor}
Assume all the hypotheses of Theorem \ref{main-thm}. Suppose in addition that there exists a standardized probability measure $\d\beta$ and a real number $p\in[1,q]$ such that
\es{\label{gen-two-point-ineq}
\|Kf\|_{L^q(\d\alpha)}\leq\|f\|_{L^p(\d\beta)} 
}
for every $f\in\C[x]$. Then $\CL$ extends to a bounded operator from $L^p(\R,\d\ga)$ to $L^q(\R,\d\ga)$ of unit norm. Moreover, $\tau=0$ (or equivalently $\CL(1)=1$) and the limit \eqref{central-limit} is also valid for all $f\in\C[x]$ and all continuous functions $g(x)$ satisfying an estimate of the form: $|g(x)| \leq A(1+|x|^A)$, for some $A>0$.
\end{theorem}

The general problem of hyper-contractive estimates for the Hermite semi-group was first partially solved by Weissler in \cite{We} and then completely solved by Epperson in \cite{Ep}. They proved that for all $p,q>0$ with $1\leq p\leq q<\infty$ the operator $T_\om$ defines a contraction from $L^p(\R,\d\ga)$ to $L^q(\R,\d\ga)$ if and only if
\es{\label{exp-om-cond}
|p-2 -\om^2(q-2)|\leq p-|\om|^2q.
}
The next corollary is an application of Theorem \ref{main-cor} to semi-groups.

\begin{corollary}[Transference Principle for Semi-groups]
Let $\d\alpha$ be a standardized probability measure and denote by $\{P_\ell(x)\}_{\ell\geq 0}$ the set of monic orthogonal polynomials associated with $\d\alpha$ (see \cite[Chapter 2]{Sz}). Define the following semi-group operator
$$
K_\om: P_\ell(x)\mapsto \om^\ell P_\ell(x)
$$
for $\om\in\C$. Then $K_\om$ satisfies all the hypotheses of Theorem \ref{main-thm} and $\CL=T_\om$ is the Centered Gaussian Operator associated with $K_\om$. 

Furthermore, if for some $\om\in\C$ and $p,q>0$ with $1\leq p\leq q<\infty$ we have an estimate of the form
\est{
\|K_\om f\|_{L^q(\d\alpha)}\leq\|f\|_{L^p(\d\alpha)} 
}
for every $f\in\C[x]$, then the operator $T_\om$ satisfies
$$
\|T_\om f\|_{L^q(\R,\d\ga)}\leq \|f\|_{L^p(\R,\d\ga)}
$$
for every $f\in \C[x]$, or equivalently, condition \eqref{exp-om-cond} must be satisfied.
\end{corollary}

In Section \ref{concluding-rem} we shall explain how our results generalize Beckner's method.

\section{Preliminaries}\label{pre}
In this section we prove some preliminary lemmas needed for the proof of Theorem \ref{main-thm}. These lemmas are mainly concerned with the representation of the operators defined in Section \ref{family-CL} in terms of Hermite polynomials.

\subsection{Hermite Polynomials}\label{herm-poly}
The Hermite polynomials $\{H_\ell(x)\}_{\ell\geq 0}$ are the orthogonal associated with the normal distribution $\d\ga$. They are recursively defined in the following way: $H_0(x)=1$, $H_1(x)=x$ and $H_\ell(x)$ is defined as the unique monic polynomial of degree $\ell$ that is orthogonal to $\{H_0,\ldots,H_{\ell-1}\}$ with respect to the inner product generated by $\d\ga$. It is known that they form a complete orthogonal basis for $L^2(\R,\d\ga)$ and are dense in $L^q(\R,\d\ga)$ for every $q\in[1,\infty)$. 

They satisfy the following recursion relation 
\begin{equation*}
H_{\ell+1}(x)=H_1(x)H_\ell(x)-\ell H_{\ell-1}(x)
\end{equation*} 
for every $\ell\geq 1$. By an application of this last formula we obtain two useful identities
\est{
\int_\R |H_\ell(x)|^2\d\ga(x)=\ell! \ \ \ \ \ \  \forall \ell \geq 0
}
and
\es{\label{herm-zero}
H_\ell(0)=\frac{(-1)^{\ell/2}\ell!}{(\ell/2)!2^{\ell/2}}
}
if $\ell$ is even. $H_\ell(0)=0$ if $\ell$ is odd.
The associated generating function is given by
\es{\label{gen-func-herm}
e^{xt-t^2/2}=\sum_{\ell\geq 0}\frac{t^\ell}{\ell!}H_\ell(x),
}
where the convergence is uniform for $t,x$ in any fixed compact set of $\C$ (see Lemma \ref{conv-lemma}). We also have the following integral representation
\es{\label{int-rep}
H_\ell(x)=\int_\R (x+iy)^\ell \d\ga(y).
}

A very important formula for our purposes is the multiplication formula below
\begin{equation}\label{mult-form}
\frac{H_\ell(x_1+\ldots+x_N)}{\ell!}=\frac{1}{N^{\ell/2}}\sum_{\ell_1+\ldots+\ell_N=\ell}\frac{H_{\ell_1}(\sqrt{N}x_1)}{\ell_1!}\ldots\frac{H_{\ell_N}(\sqrt{N}x_N)}{\ell_N!},
\end{equation}
which holds for every $(x_1,\ldots,x_N)\in\C^N$. This last formula can be deduced by using formula \eqref{int-rep} and the fact that $\d\ga(x)=\d\ga^N(\sqrt{N}x)$ for every $N>0$ (see the notation Section \ref{notation}).
 
All these facts about Hermite Polynomials can be found in \cite[Chapter 5]{Sz}.

\subsection{Convergence and Representation lemmas}
The proof of Theorem \ref{main-thm} relies on the formal representation in terms of Hermite polynomials of an operator $\CL\in\mathfrak{C}$. The next lemma deals first with the convergence issues. We begin by compiling useful estimates.

\begin{lemma}\label{estimation-lemma-1}
We have the following estimates:
\begin{enumerate}
\item For any real numbers $q\geq 1$ and $B>0$ we have
\begin{equation*}
\lim_{N\to\infty}\frac{B^N}{N!}\left\||x+iy|^Ne^{B|x|} \right\|_{L^q(\R^2,\d\ga(x)\times\d\ga(y))} =0.
\end{equation*}
\item For every $t,x\in\C$ we have
$$
\sum_{\ell=0}^L\left|\frac{t^\ell}{\ell!}H_\ell(x)\right| \leq e^{|tx|}\int_{\R}e^{|ty|}\d\ga(y).
$$

\item For every $t,x\in\C$ we have
$$
\sum_{\ell> L}\left|\frac{t^\ell}{\ell!}H_\ell(x)\right| \leq e^{|tx|}\int_\R \frac{|t(x+iy)|^{L+1}}{(L+1)!}\d\ga(y).
$$
\end{enumerate}
\end{lemma}

\begin{proof}
Estimates $(2)$ and $(3)$ are consequences of integral formula \eqref{int-rep} and the following inequalities respectively
\es{\label{exp-ineq-large}
\sum_{\ell=0}^L\frac{s^\ell}{\ell!}\leq e^{s} \ \ \ \ \text{and} \ \ \ \ \ \sum_{\ell>L} \frac{s^\ell}{\ell!} \leq e^{s}\frac{s^{L+1}}{(L+1)!},\ \ \ \ (s\geq 0).
}
Using the following inequalities 
\est{
(a+b)^t \leq 2^{t-1}(a^t+b^t) \ \ \ \ \ \text{and} \ \ \ \ ab \leq \frac{2a^{3/2}}{3}+\frac{b^3}{3}, \ \ \ \ (a,b\geq 0, \ t\geq 1)
} 
we deduce that
\est{
\int_\R\int_\R |x+iy|^{Nq}e^{Bq|x|}\d\ga(x) \d\ga(y)& \< 2^{Nq}\left(1+\int_{\R}|x|^{3Nq/2}\d\ga(x)\right)\\
& =  2^{Nq}\left(1+\pi^{-1/2}2^{3Nq/4}\Gamma(3Nq/4+1/2)\right)\\
& \< 4^{Nq}\left(1+\Gamma(3Nq/4+1/2)\right),
}
where the implied constants depend only on $B$ and $q$. Using Stirling's formula
$$
\Gamma(1+t) \sim \sqrt{2\pi}\,t^{t+1/2}e^{-t}, \ \ \ \ t\to\infty
$$
the limit $(1)$ follows. This completes the proof.
\end{proof}

Now we prove a useful inequality.

\begin{lemma}\label{conv-lemma}
Let $\om,\la\in\C$. Then for every $L'<L$ and every $t,x\in\C$ we have
\est{
\bigg| \exp\big[&x\om\la  t-(1-\la^2+\om^2\la^2)t^2/2\big] -  \sum_{\ell=0}^L\frac{t^\ell}{\ell!}T_\om S_\la(H_\ell)(x)\bigg|
\\ & \leq  \bigg\{ \int_\R \frac{|t(x+iy)|^{L+1}}{(L+1)!}\d\ga(y) + \frac{|(\la^2-1)t^2/2|^{\lfloor (L-L')/2\rfloor+1}}{(\lfloor (L-L')/2\rfloor+1)!}\int_\R e^{|\om \la t y|}\d\ga(y)
\\ & \ \ \ \ \ \ \ \ + \int_\R \frac{|\om\la t(x+iy)|^{L'+1}}{(L'+1)!}\d\ga(y)\bigg\}\exp\left[|(\la^2-1)t^2/2|+(1+|\om\la|)|tx|\right].
}
\end{lemma}

\begin{proof}
Using the generating function \eqref{gen-func-herm} one can deduce that
$$
H_\ell(\la x) = \sum_{k=0}^\ell \binom \ell k \la^k (1-\la^2)^{\tfrac{\ell-k}{2}}H_{\ell-k}(0)H_k(x).
$$
Now, we can use \eqref{herm-zero} to obtain
\est{
\sum_{\ell=0}^L\frac{t^\ell}{\ell!}T_\om S_\la(H_\ell)(x) 
& = \sum_{k=0}^L \left(\sum_{\ell=0}^{\left \lfloor{\tfrac{L-k}{2}}\right \rfloor }\frac{((\la^2-1)t^2/2)^\ell}{\ell!}\right)\frac{(\om\la t)^k}{k!}H_k(x) \\
& = e^{(\la^2-1)t^2/2}\sum_{k=0}^L\frac{(\om\la t)^k}{k!}H_k(x)\\& \,\,\,\,\,\,\,\, -\sum_{k=0}^L \left(\sum_{\ell>\left \lfloor{\tfrac{L-k}{2}}\right \rfloor }\frac{((\la^2-1)t^2/2)^\ell}{\ell!}\right)\frac{(\om\la t)^k}{k!}H_k(x).
}
Thus, we have
\est{
\exp[x\om\la & t-(1-\la^2+\om^2\la^2)t^2/2] -\sum_{\ell=0}^L\frac{t^\ell}{\ell!}T_\om S_\la(H_\ell)(x) \\
& = e^{(\la^2-1)t^2/2}\sum_{\ell>L}\frac{t^\ell}{\ell!}H_\ell(x) + \sum_{k=0}^L \left(\sum_{\ell>\left \lfloor{\tfrac{L-k}{2}}\right \rfloor }\frac{((\la^2-1)t^2/2)^\ell}{\ell!}\right)\frac{(\om\la t)^k}{k!}H_k(x) \\
& =: \I_1(t,x,L) + \I_2(t,x,L)
}
We now estimate quantities $\I_1$ and $\I_2$. Using the estimate $(3)$ of Lemma \ref{estimation-lemma-1} we obtain 
\es{\label{eq-4}
|\I_1(t,x,L)|\leq e^{|(\la^2-1)t^2/2|+|tx|} \int_\R \frac{|t(x+iy)|^{L+1}}{(L+1)!}\d\ga(y).
}
Now, we split quantity $\I_2$ into two parts
\est{
\I_2(t,x,L) = \J_1(t,x,L',L) + \J_2(t,x,L',L),
}
where $\J_1(t,x,L',L)$ denotes the sum from $k=0$ to $k=L'$ and $\J_2(t,x,L',L)$ the sum from $k=L'+1$ to $L$. Applying estimate $(2)$ of Lemma \eqref{estimation-lemma-1} and inequality \eqref{exp-ineq-large} we obtain
\es{\label{eq-5}
\J_1(t,x,L',L) & \leq \left(\sum_{\ell=0}^{L'} \frac{|\om\la t|^k}{k!}|H_k(x)|\right) \sum_{\ell>\left \lfloor{\tfrac{L-L'}{2}}\right \rfloor }\frac{|(\la^2-1)t^2/2|^\ell}{\ell!} \\ 
& \leq e^{|(\la^2-1)t^2/2|+|\om\la tx|}\frac{|(\la^2-1)t^2/2|^{\lfloor (L-L')/2\rfloor+1}}{(\lfloor (L-L')/2\rfloor+1)!}\int_\R e^{|\om \la t y|}\d\ga(y).
}
By a similar method we obtain
\es{\label{eq-6}
\J_2(t,x,L',L) \leq e^{|(\la^2-1)t^2/2|+|\om \la tx|} \int_\R \frac{|\om\la t(x+iy)|^{L'+1}}{(L'+1)!}\d\ga(y).
}
The lemma follows from \eqref{eq-4}, \eqref{eq-5} and \eqref{eq-6}.
\end{proof}

\noindent {\it Remark.}
Notice that by taking $L'=\lfloor L/2 \rfloor$ the previous lemma implies that
\est{
\lim_{L\to\infty} \sum_{\ell=0}^L\frac{t^\ell}{\ell!}T_\om S_\la(H_\ell)(x) = \exp\left[x\om\la  t-(1-\la^2+\om^2\la^2)t^2/2\right]
}
where the convergence is uniform for $t$ and $x$ in any fixed compact set of $\C$.
\smallskip

Let $\CL\in\mathfrak{C}$ with $\CL=M_\tau T_\omega S_\lambda$. Since $\Rep \, \tau> -1$, we can easily see that $\CL(\C[x])\subset L^{1+\varepsilon}(\R,\d\ga)$ for some small $\varepsilon>0$. Therefore, the following coefficients
\es{\label{form-CL}
c_{\ell,m}=\int_{\R}\CL(H_\ell)(x)H_m(x)\d\ga(x)
}
are well defined and
$$
\CL(H_\ell)(x) = \sum_{m\geq 0} \frac{c_{\ell,m}}{m!}H_m(x)
$$
in the $L^2(\R,\d\ga)-$sense if $\Rep\,\tau>-1/2$. The next lemma gives an exact analytic expression for these coefficients. Below the operation $\land$ represents the minimum between two given numbers and $\lor$ represents the maximum.

\begin{lemma}\label{rep-lemma}
Let $\tau,\om,\la\in\C$ with $\re\tau>-1$. Then
\es{\label{coeff-CL}
\frac{c_{\ell,m}}{\ell!m!}=\sum_{\stackrel{n=0}{n \,\, \text{even}}}^{\ell\land m}  \frac{\left(\tfrac{-\tau}{\tau+1}\right)^{\tfrac{\ell\lor m-\ell+n}{2}}(-a)^{\tfrac{\ell\lor m-m+n}{2}}b^{\ell\land m-n}}{2^{\tfrac{|\ell-m|}{2}+n}(n/2)!\left(\frac{|m-\ell|+n}{2}\right)!(\ell\land m-n)!}
}
if $\ell+m$ is even and $c_{\ell,m}=0$ if $\ell+m$ is odd. The quantities $a$ and $b$ are given by
\es{\label{ab-eq}
a =1-\la^2 + \lambda^2\om^2\frac{\tau}{\tau+1}, \ \ \ \ \ \ b =\frac{\la \om}{\tau+1}.
}
\end{lemma}

\begin{proof}
{\it Step 1.}
Define for every $N>0$ the following function
\es{\label{eq-7.5}
F_N(s,t)=\sum_{\ell,m=0}^Nc_{\ell,m}\frac{t^\ell s^m}{\ell!m!} = \int_{\R}M_\tau T_\om S_\la\left(\sum_{\ell=0}^N \frac{t^\ell}{\ell!}H_\ell \right)(x)\left(\sum_{m=0}^N \frac{s^m}{m!}H_m(x)\right)\d\ga(x)
}
for every $s,t\in\C$. We claim that if $\Rep \, \tau \geq 0$ and $q\in[1,\infty)$ then
\es{\label{eq-8}
\bigg\| \sqrt{1+\tau}\exp[x\om\la  t-&(1-\la^2+\om^2\la^2)t^2/2-\tau x^2/2] \\ & -  \sum_{\ell=0}^L\frac{t^\ell}{\ell!}M_\tau T_\om S_\la(H_\ell)(x)\bigg\|_{L^q(\d\ga(x))}
}
converges to zero uniformly in the variable $t$ in any fixed compact set of $\C$. 

Assuming the claim is true, we prove the lemma. First we deal with the case $\Rep\,\tau \geq 0 $. In this case, by an application of H\"older's inequality in \eqref{eq-7.5} we deduce that
\est{
F(s,t)& :=\lim_{N\to \infty}F_N(s,t)\\ & = \sqrt{1+\tau}\int_{\R}\exp\left[x(\om\la t+s)-(1-\la^2+\om^2\la^2)t^2/2-s^2/2-\tau x^2/2\right]\d\ga(x),
}
where the limit is uniform in compact sets of $\C$ in the variables $s$ and $t$. Using the following identity
$$
\int_{\R}e^{-A(x-B)^2}\dx = \sqrt{\frac{\pi}{A}}
$$
which holds for every $A,B\in\C$ with $\Rep\,A>0$ we conclude that
\est{
F(s,t) = \exp\left[-at^2/2-\frac{\tau}{\tau+1} s^2/2+b ts\right],
}
with $a$ and $b$ given by \eqref{ab-eq}. We can now use the generating function \eqref{gen-func-herm} to obtain
\est{
F(s,t) & =\left( \sum_{i\geq 0} \frac{a^{i/2}t^i}{i!}H_i(0) \right)\left(\sum_{j\geq 0} \left(\frac{\tau}{\tau+1}\right)^{j/2}\frac{s^j}{j!}H_j(0)\right) \left( \sum_{k\geq 0}\frac{(bts)^k}{k!}\right)\\
& =  \sum_{i,j,k\geq 0} \frac{t^{i+k}s^{j+k}}{i!j!k!}a^{i/2}\left(\frac{\tau}{\tau+1}\right)^{j/2}b^{k}H_i(0)H_j(0) \\
& = \sum_{\ell,m\geq 0} t^\ell s^m\sum_{n=0}^{\ell\land m} \frac{\left(\tfrac{\tau}{\tau+1}\right)^{\tfrac{\ell\lor m-\ell+n}{2}}a^{\tfrac{\ell\lor m-m+n}{2}}b^{\ell\land m-n}}{n!(|m-\ell|+n)!(\ell\land m-n)!}H_{n}(0)H_{|\ell-m|+n}(0),
}
where in the last identity we made the following change of variables: $\ell=i+k$ and $m=j+k$. 

Using identity \eqref{herm-zero} in conjunction with the fact that $F_N(s,t)$ converges uniformly in compact sets to $F(s,t)$ we deduce that the coefficients of their Taylor series must match and thus the representation \eqref{coeff-CL} follows for $\Rep \,\tau \geq 0$. However, expressions \eqref{coeff-CL} and \eqref{form-CL} clearly define analytic functions in the variable $\tau$ for $\Rep \,\tau >-1$. Thus, by analytic continuation, \eqref{coeff-CL} also holds for $\Rep \,\tau >-1$.
\smallskip

{\it Step 2.}
It remains to prove the claim stated in \eqref{eq-8} for $\Rep\,\tau \geq 0$.
Let $t_0>0$ and assume that $|t|\leq t_0$. Using Lemma \ref{conv-lemma} and Jensen's inequality we obtain
\es{\label{eq-10}
\bigg\| \sqrt{1+\tau}\exp & \left[x\om\la  t-\tfrac{1-\la^2+\om^2\la^2}{2}t^2-\tfrac{\tau}{2} x^2\right] -  \sum_{\ell=0}^L\frac{t^\ell}{\ell!}M_\tau T_\om S_\la(H_\ell)(x)\bigg\|_{L^q(\d\ga(x))}
\\ & \leq |1+\tau|^{1/2}B\Bigg\{\frac{B^{L+1}}{(L+1)!} \bigg\| |x+iy|^{L+1}e^{B|x|}\bigg\|_{L^q(\d\ga(x)\times\d\ga(y))} \\
& \ \ \ \ \ \ \ \ \ \ \ \ \ \ \ \ \ \ \ \ + \frac{B^{\lfloor (L-L')/2\rfloor+1}}{(\lfloor (L-L')/2\rfloor+1)!}\bigg\| e^{B(|y|+|x|)}\bigg\|_{L^q(\d\ga(x)\times\d\ga(y))}\\
& \ \ \ \ \ \ \ \ \ \ \ \ \ \ \ \ \ \ \ \ \ \ \ + \frac{B^{L'+1}}{(L'+1)!} \bigg\| |x+iy|^{L'+1}e^{B|x|}\bigg\|_{L^q(\d\ga(x)\times\d\ga(y))} \Bigg\}
}
for all $L'<L$, where $B$ is a constant which depends only on $|\la|,|\om|$ and $t_0$. Choosing $L'=\lfloor L/2\rfloor$ and using item $(1)$ of Lemma \ref{estimation-lemma-1} one can easily see that the right hand side of \eqref{eq-10} converges to zero when $L\to \infty$. This finishes the proof.
\end{proof}

\section{Proofs of the Main Results}\label{proof-main-res}

\begin{proof}[Proof of Theorem \ref{main-thm}]
The main ingredient of the proof is the multiplication formula \eqref{mult-form}. By the fact the any polynomial can be uniquely written as a linear combination of Hermite polynomials and by Lemma \ref{rep-lemma} it is sufficient to prove that
$$
c_{\ell,m}(N):=\int_{\R^N}K_N([H_\ell]_+)(\bo x)[H_m]_+(\bo x)\d\alpha_N(\sqrt{N}\bo x) \to c_{\ell,m}, \ \ \ \ \ N\to\infty
$$
where $c_{\ell,m}$ is given by \eqref{coeff-CL} with the parameters $\tau,\om,\la$ given by equations $(i),(ii)$ and $(iii)$ in Theorem \ref{main-thm}. Applying identity \eqref{mult-form} we obtain
\est{
c_{\ell,m}(N)  = \frac{\ell!m!}{N^{\tfrac{\ell+m}{2}}}\sum_{\stackrel{\ell_1+\ldots+\ell_N=\ell}{m_1+\ldots+m_N=m}} \frac{K_{\ell_1,m_1}}{\ell_1!m_1!}\ldots\frac{K_{\ell_N,m_N}}{\ell_N!m_N!}.
}
By doing a change of variables that counts the number of appearances of each term $\frac{K_{i,j}}{i!j!}$ we obtain that
\est{
c_{\ell,m}(N) = \frac{\ell!m!}{N^{\tfrac{\ell+m}{2}}} \sum_{[P_{i,j}]} \frac{N!}{\prod_{i,j}P_{i,j}!}\prod_{i,j} \left(\frac{K_{i,j}}{i!j!}\right)^{P_{i,j}}
}
where the last sum is over the subset of matrices $[P_{i,j}]$, $i=0,\ldots,\ell$ and $j=0,\ldots,m$ with non-negative integer entries satisfying the conditions below:
\begin{enumerate}
\item[$(I)$] $\sum_{i,j} iP_{i,j} =\ell$;
\item[$(II)$] $\sum_{i,j} jP_{i,j} =m$;
\item[$(III)$] $\sum_{i,j} P_{i,j} =N$.
\end{enumerate}

These conditions imply that
$$
P_{i,j}\leq \max\{\ell,m\}
$$
if $(i,j)\neq (0,0)$ and 
$$
N\geq P_{0,0}\geq N-\max\{\ell,m\}[(\ell+1)(m+1)-1].
$$ 
Thus, the subset of matrices determined by $(I)$--$(III)$ is finite and the number of elements does not depend on $N$. Also, since $K_{0,0}=1$ we obtain
\es{\label{eq-2}
c_{\ell,m}(N)=\frac{\ell!m!}{N^{\tfrac{\ell+m}{2}}} \sum_{[P_{i,j}]} \left\{\frac{N!}{(N-\sum'_{i,j}P_{i,j})!\prod_{i,j}'P_{i,j}!}\prod'_{i,j} \left(\frac{K_{i,j}}{i!j!}\right)^{P_{i,j}}\right\},
}
where the symbols $\prod'$ and $\sum'$ mean that the term $(i,j)=(0,0)$ is excluded. We also obtain that for every $[P_{i,j}]$ satisfying $(I)$--$(III)$ we have
$$
\frac{N!}{(N-\sum'_{i,j}P_{i,j})!\prod_{i,j}'P_{i,j}!} \sim \frac{N^{\sum'_{i,j}P_{i,j}}}{\prod_{i,j}'P_{i,j}!}
$$
when $N\to\infty$ (the symbol $\sim$ means that the quotient goes to $1$ when $N\to\infty$).

We now investigate the possible values for $\sum'_{i,j}P_{i,j}$.  Notice that if $P_{0,1}$ or $P_{1,0}$ is not zero then the quantity in the brackets at \eqref{eq-2} is zero because $K_{0,1}=K_{1,0}=0$. If $P_{0,1}=P_{1,0}=0$, then by equations $(I)$ and $(II)$ we conclude that 
$$
\frac{\ell+m}{2} ={\sum_{i,j}}'\frac{i+j}{2}P_{i,j}\geq {\sum_{i,j}}'P_{i,j}
$$
with equality occurring if and only if $\ell+m=2(P_{0,2}+P_{2,0}+P_{1,1})$, $\ell+m$ is even and $P_{i,j}=0$ if $(i,j)\notin\{(0,2),(2,0),(1,1),(0,0)\}$. 

We conclude that the limit of \eqref{eq-2} when $N\to\infty$ is zero if $\ell+m$ is odd and is equal to
\es{\label{eq-3}
\ell!m!\sum \frac{K_{0,2}^{P_{0,2}}K_{2,0}^{P_{2,0}}K_{1,1}^{P_{1,1}}}{2^{P_{0,2}+P_{2,0}}P_{0,2}!P_{2,0}!P_{1,1}!}
}
if $\ell+m$ is even, where the above sum is over the set of non-negative integers $P_{0,2},P_{2,0},P_{1,1}$ satisfying
\begin{enumerate}
\item[$(IV)$] $2P_{2,0}+P_{1,1}=\ell$,
\item[$(V)$] $2P_{0,2}+P_{1,1}=m$.
\end{enumerate}
Depending on whether $\ell$ is greater than $m$ or not, one can do an appropriate change of variables (for instance if $m\geq \ell$ choose $n=2P_{2,0}$) to deduce that \eqref{eq-3} equals to
$$
\ell!m!\sum_{\stackrel{n=0}{n\,\text{even}}}^{\ell\land m} \frac{K_{1,1}^{\ell\land m-n}K_{2,0}^{\tfrac{\ell\lor m-m+n}{2}}K_{0,2}^{\tfrac{\ell\lor m-\ell+n}{2}}}{2^{\tfrac{|\ell-m|}{2}+n}(\ell\land m-n)!(n/2)!\left(\tfrac{|\ell-m|+n}{2}\right)!}.
$$
Finally, we can apply Lemma \ref{rep-lemma} to conclude that the above quantity equals to
$$
\int_{\R}\CL(H_\ell)(x)H_{m}(x)\d\ga(x)
$$ 
if $K_{0,2}=-\tau/(1+\tau)$, $K_{2,0}=\la^2-1-\la^2\om^2\tau/(\tau+1)$ and $K_{1,1}=\la\om/(1+\tau)$. This finishes the proof.
\end{proof}

We now prove our second main result.

\begin{proof}[Proof of Theorem \ref{main-cor}]
First, we claim that for every $N>0$ the operator $K_N$ defined in Section \ref{family-CL} satisfies 
\es{\label{eq-12}
\left\|K_N(f)(\bo x)\right\|_{L^q(\R^N,\d\alpha_N(\sqrt{N}\bo x))} \leq \left\|f(\bo x)\right\|_{L^p(\R^N,\d\beta_N(\sqrt{N}\bo x))}
}
for every polynomial $f\in\C[x_1,\ldots,x_N]$ (recall the notation in Section \ref{notation}). Denoting by 
$$
g(x_1,\ldots,x_{N-1},y_N)=K_{y_1}K_{y_2}\ldots K_{y_{N-1}}\left[f(\bo y/\sqrt{N})\right](\sqrt{N}x_1,\ldots,\sqrt{N}x_{n-1}),
$$
where $K_{y_j}$ denotes the restriction to the $y_j$ variable of the operator $K$, we conclude that
$$
K_N(f)(x_1,\ldots,x_N)=K_{y_N}[g(x_1,\ldots,x_N,y_N/\sqrt{N})](\sqrt{N}x_N).
$$
We obtain
\begin{align*}
\begin{split}
& \|K_Nf(\bo x)\|_{L^q(\d\alpha(\sqrt{N}x_1)\times\ldots\times\d\alpha(\sqrt{N}x_N))} \\ &= \|\|K_{y_N}[g(x_1,\ldots,x_{N-1},\tfrac{y_N}{\sqrt{N}})](\sqrt{N}x_N)\|_{L^q(\d\alpha(\sqrt{N}x_N))}\,\|_{L^q(\d\alpha(\sqrt{N}x_1)\times\ldots\times\d\alpha(\sqrt{N}x_{N-1}))}
\end{split} \\ &\leq \|\|g(x_1,\ldots,x_{N-1},y_N)]\|_{L^p(\d\beta(\sqrt{N}y_N))}\,\|_{L^q(\d\alpha(\sqrt{N}x_1)\times\ldots\times\d\alpha(\sqrt{N}x_{N-1}))} \\ &\leq \|\|g(x_1,\ldots,x_{N-1},y_N)]\,\|_{L^q(\d\alpha(\sqrt{N}x_1)\times\ldots\times\d\alpha(\sqrt{N}x_{N-1}))}\,\|_{L^p(\d\beta(\sqrt{N}y_N))},
\end{align*}
where the second inequality is Minkowski's inequality since $q\geq p$. We now can apply the same argument to estimate the quantity
$$
\|g(x_1,\ldots,x_{N-1},y_N)\,\|_{L^q(\R^{N-1},\d\alpha(\sqrt{N}x_1)\times\ldots\times\d\alpha(\sqrt{N}x_{N-1}))}
$$
for fixed $y_N$ and conclude by induction that \eqref{eq-12} is valid (see also \cite[Lemma 2]{Be}). 

Now, let $\CL\in\mathfrak{C}$ be the Centered Gaussian operator associated with $K$ and $f\in\C[x]$ with $\|f\|_{L^p(\R,\d\ga)}=1$. Since $\C[x]$ is dense in $L^q(\R,\d\ga)$, for every $\varepsilon>0$ we can find $g\in\C[x]$ with $\|g\|_{L^{q'}(\R,\d\ga)}=1$ such that 
$$
\|\CL(f)\|_{L^q(\R,\d\ga)} \leq \left|\int_\R \CL(f)(x)g(x)\d\ga(x)\right| + \varepsilon.
$$
However, for $N$ sufficiently large we have 
\es{\label{eq-13}
\left|\int_\R \CL(f)(x)g(x)\d\ga(x)\right| & \leq \left|\int_{\R^N} K_N(f_+)(\bo x)g_+(\bo x)\d\alpha_N(\sqrt{N}\bo x)\right| + \varepsilon \\
& \leq \left\|K_N(f_+)(\bo x)\right\|_{L^q(\R^N,\d\alpha_N(\sqrt{N}\bo x))} \left\|g_+(\bo x)\right\|_{L^{q'}(\R^N,\d\alpha_N(\sqrt{N}\bo x))} + \varepsilon \\
& \leq \left\|f_+(\bo x)\right\|_{L^p(\R^N,\d\beta_N(\sqrt{N}\bo x))} \left\|g_+(\bo x)\right\|_{L^{q'}(\R^N,\d\alpha_N(\sqrt{N}\bo x))}+\varepsilon \\
 & = \left\|f(x)\right\|_{L^p(\R,\d\beta^N(\sqrt{N}x))} \left\|g(x)\right\|_{L^{q'}(\R,\d\alpha^N(\sqrt{N}x))} + \varepsilon,
}
where the second inequality is H\"olçder's inequality and the third one is due to \eqref{eq-12}. 

Since $\d\alpha$ is a standardized probability measure, Fatou's lemma for weakly convergent probabilities \cite[Theorem 1.1]{FKZ} together with the convergence of the absolute moments in the Central Limit Theorem \cite[Theorem 2]{Ba} imply that
\es{\label{CLT-strong}
\lim_{N\to\infty} \int_\R h(x)\d\alpha^N(\sqrt{N}x) = \int_\R h(x)\d\ga(x)
}
for every continuous function $h(x)$ satisfying an estimate of the form $|h(x)|\leq A(1+|x|^A)$, for some $A>0$. Thus, we conclude that the right hand side of \eqref{eq-13} converges to $1+\varepsilon$ when $N\to \infty$. By the arbitrariness of $\varepsilon>0$ we conclude that $\|\CL(f)\|_{L^q(\d\ga)} \leq 1$. By the density of $\C[x]$ in $L^p(\R,\d\ga)$ for finite $p\geq 1$, we conclude that $\CL$ extends to a bounded linear operator of norm not greater than one.

Now, observe that since $K_{0,0}=1$ we have
$$
\int_{\R^N}K_N(1)(\bo x)\d\alpha_N(\sqrt{N}\bo x)=1
$$
for every $N>0$. Thus, we obtain
$$
1=\int_{\R}\CL(1)(x)\d\ga(x) \leq \|\CL(1)\|_{L^q(\R,\d\ga)} \leq 1.
$$
This implies that $|\CL(1)(x)|=1$ for every real $x$. We conclude that $\tau=0$ or, equivalently, $\CL(1)=1$.

Now, let $f\in\C[x]$ and let $g(x)$ be a continuous function satisfying an estimate of the form $|g(x)|\leq A(1+|x|^A)$.
Given $\varepsilon>0$, take $h\in\C[x]$ such that $\|g-h\|_{L^{q'}(\R,\d\ga)}<\varepsilon$. By estimate \eqref{eq-12} and Holder's inequality we conclude that
\est{
\bigg|\int_\R &\CL(f)(x)g(x)\d\ga(x) -  \int_{\R^N} K_N(f_+)(\bo x)g_+(\bo x)\d\alpha_N(\sqrt{N}\bo x)\bigg| \\
& \leq  \varepsilon\|\CL(f)\|_{L^q(\d\ga)} + \|f(x)\|_{L^p(\d\alpha^N(\sqrt{N}x))}\|g(x)-h(x)\|_{L^{q'}(\d\alpha^N(\sqrt{N}x))}.
}
We can now use the Central Limit Theorem as stated in \eqref{CLT-strong} to obtain that
\est{
\limsup_{N\to\infty} \bigg|\int_\R \CL(f)(x)g(x)\d\ga(x) -  \int_{\R^N} &K_N(f_+)(\bo x)g_+(\bo x)\d\alpha_N(\sqrt{N}\bo x)\bigg| \\ & \leq  \varepsilon\left(\|\CL(f)\|_{L^q(\d\ga)} + \|f(x)\|_{L^p(\d\ga)}\right).
}
The proof is complete once we let $\varepsilon\to 0$.
\end{proof}

\section{Concluding Remarks}\label{concluding-rem}

\subsection{Beckner's Method}
We now explain how our results generalize the procedure used by Beckner in \cite{Be} to prove the sharp form of the Hausdorff-Young inequality. 

First, using formula \eqref{gen-func-herm} he deduced that
$$
\F: H_n(\sqrt{2\pi p}x)e^{-\pi x^2}\mapsto \om^n H_n(\sqrt{2\pi p}x)e^{-\pi x^2}
$$
where $\om=i\sqrt{p-1}$. Then, by a change of variables he showed that the sharp Hausdorff-Young inequality for $1<p\leq 2$ (with Gaussian being maximizers) is equivalent to 
$$
\|T_\om f\|_{L^{p'}(\d\ga)} \leq \|f\|_{L^{p}(\d\ga)}
$$
for every $f\in\C[x]$.

Secondly, by choosing the measure 
$$
\d\nu=\frac{\delta_{-1}+\delta_{1}}{2}
$$
and the operator
$$
K_\om(f)(x)=\int_\R f(x)\d\nu(x) + \om\int_\R xf(x)\d\nu(x),
$$
Beckner derived a famous two-point inequality \cite[Lemma 1]{Be}, which is exactly an estimate of the form \eqref{gen-two-point-ineq} with $q=p'$ and $\d\alpha=\d\beta=\d\nu$. 

Finally, he showed the convergence result of Theorem \ref{main-thm} by a different argument, in which he exploited a special relation between symmetric functions and Hermite polynomials.

\subsection{Gaussian Kernels}
In this section we show that every operator $\CL\in\mathfrak{C}$ is given by a Gaussian Kernel. If $\CL=M_\tau T_\om S_\la$ then
$$
\CL f(x) = \int_{\R} \CL(x,y)f(y)\d\ga(y)
$$
for every $f\in\C[x]$ where
\est{
\CL(x,y) = \sqrt{\frac{1+\tau}{\la^2(1-\om^2)}}\exp\left[-\frac{\tau+(1-\tau)\om^2}{2(1-\om^2)}x^2-\frac{1-\la^2(1-\om^2)}{2\la^2(1-\om^2)}y^2+\frac{\om xy}{\la(1-\om^2)}\right].
}
This kernel can be calculated by using the fact that the operator $T_\om$ is given by the following Mehler kernel (see \cite[p. 163]{Be})
$$
T_\om(x,y) = \frac{1}{\sqrt{1-\om^2}}\exp\left[-\frac{\om^2(x^2+y^2)}{2(1-\om^2)}+\frac{\om x y}{1-\om^2}\right].
$$
Therefore, by inverting the system of equations below
\est{
A=\frac{\tau+(1-\tau)\om^2}{(1-\om^2)}, \ \ \ \ B = \frac{1-\la^2(1-\om^2)}{\la^2(1-\om^2)} \ \ \text{and} \ \ C= \frac{\om}{\la(1-\om^2)}
}
we conclude that the class of Centered Gaussian Operators $\mathfrak{C}$ coincides with the class of operators given by centered Gaussian kernels of the following form
\est{
G(x,y)=\exp[-(A/2)x^2-(B/2)y^2+Cxy+D].
}

An interesting problem within this theory consists of studying for which parameters $A,B$ and $C$ an operator of this form is bounded from $L^p(\R,\d \ga)$ to $L^q(\R,\d \ga)$ and, if that is the case, classify the set of maximizers. This problem was studied by Lieb \cite{Li} where he showed that (in a much more general context) in most cases if $\CL$ is bounded from $L^p(\R,\d\ga)$ to $L^q(\R,\d\ga)$ then it is a contraction.

\section*{Acknowledgments}
I am deeply grateful to William Beckner for encouraging me to work on this problem and for all the fruitful discussions on the elaboration of this paper. The author also acknowledges the support from CNPq-Brazil (Science Without Borders program) visiting scholar fellowship 201697/2014-9 at The University Of Texas at Austin.


\begin{thebibliography}{99}
\bibitem{Be}
W. Beckner, 
\newblock Inequalities in Fourier Analysis,
\newblock Annals of Mathematics, 102 (1975), 159-182.

\bibitem{DGS}
E. Davies, L. Gross and B. Simon,
\newblock Hypercontractivity: A Bibliographic Review.
\newblock 

\bibitem{Ep}
J. Epperson,
\newblock The Hypercontractive Approach to Exactly Bounding an Operator with Complex Gaussian Kernel,
\newblock Journal of Functional Analysis 87, 1-30 (1989).

\bibitem{FKZ}
E. Feinberg, P. Kasyanov, and N. Zadoianchuk,
\newblock Fatou's Lemma for Weakly Converging Probabilities,
\newblock Theory Probab. Appl., 58(4), 683–689 (2014).

\bibitem{Gr}
L. Gross,
\newblock Logarithmic Sobolev Inequalities,
\newblock American Journal of Mathematics, Vol. 97, No. 4, 1061-1083.

\bibitem{Li}
E. Lieb,
\newblock Gaussian Kernels Have Only Gaussian Maximizers,
\newblock Invent. Math. 102, 179-208 (1990).

\bibitem{Ba}
B. Von Bahr,
\newblock On the Convergence of Moments in the Central Limit Theorem,
\newblock Ann. Math. Statist. Volume 36, Number 3 (1965), 808-818.

\bibitem{Sz}
G. Szeg\"o,
\newblock {\it Orthogonal Polynomials},
\newblock American Mathematical Society Colloquium Publications Volume XXIII, Fourth Edition, 1975.

\bibitem{We}
F. Weissler,
\newblock Two-Point Inequalities, the Hermite Semigroup, and the Gauss-Weierstrass Semigroup,
\newblock Journal of Functional Analysis 32, 102-121 (1979).
\end{thebibliography}
\end{document}